\documentclass[a4paper, 11pt, english, oneside, reqno]{article}

\usepackage[top=1.25in, bottom=1.25in, left=1.25in, right=1.25in]{geometry}
\usepackage[utf8]{inputenc}
\usepackage{amsfonts, amssymb, amsmath, amsthm, wasysym}
\usepackage{mathrsfs}
\usepackage{graphicx}
\usepackage[dvipsnames]{xcolor}
\definecolor{blue-green}{rgb}{0.0, 0.85, 0.85} % deepskyblue
\definecolor{some-blue}{rgb}{0.12, 0.56, 1.0} %dodgerblue
\definecolor{chocolate}{rgb}{0.48, 0.25, 0.0}

\usepackage[toc,page]{appendix}
\usepackage{caption}

\usepackage{transparent}
\usepackage{eepic}
\usepackage{tocbibind} %la bibliografía aparece en la table of contents

\usepackage{tikz} %Dibujos en LaTeX desde mathcha.io

\usepackage{accents}
\usepackage{setspace}
\usepackage{easyReview}
\usepackage{footnote}
\usepackage{etoolbox}

\usepackage[colorlinks = true]{hyperref}
\hypersetup{
	pdfauthor={Javier Minguillón, Fernando Soria, Ana Vargas},
	pdftitle={On the Pointwise Convergence of Solutions to the Schrödinger Equation Along Certain Highly Tangential Curves},
	pdfkeywords={Schrödinger equation, Schrödinger maximal function, Almost everywhere convergence, 
		Tangential convergence, Fourier analysis},
	pdfcreator={LaTeX with hyperref},
	pdfproducer={pdfLaTeX}
	colorlinks=true,
	linkcolor={blue!80!black},
	citecolor={blue!50!black},
	urlcolor={blue!50!black}
}
\usepackage{cleveref}

\def\R{\mathbb R}
\def\Z{\mathbb Z}
\def\N{\mathbb N}

\def\C{\mathbb C}

\newcommand{\supp}{\text{\normalfont supp}}

\usepackage{accents}

\newcommand{\lessssim}{\lesssim C_\varepsilon R^\varepsilon}
\newcommand{\defeq}{=} 
\newcommand{\len}{\text{length}}

\newcommand{\calpha}{D}
\newcommand{\schtf}[1]{e^{it\Delta}f \left(#1\right)}
\newcommand{\dxt}{dx 
	dt}
\newcommand{\esnalpha}{s}%{s(n,a)}
\newcommand{\tildeq}{Q}
\usepackage{xparse}

\let\latexchi\chi
\makeatletter
\renewcommand\chi{\@ifnextchar_\sub@chi\latexchi}
\newcommand{\sub@chi}[2]{% #1 is _, #2 is the subscript
	\@ifnextchar^{\subsup@chi{#2}}{\latexchi^{}_{#2}}%
}
\newcommand{\subsup@chi}[3]{% #1 is the subscript, #2 is ^, #3 is the superscript
	\latexchi_{#1}^{#3}%
}
\makeatother

\newcommand{\scalarp}[1]{ 
}

\newcommand{\verluegointerno}[1]{{\color{chocolate}
}}

\newcommand{\verluegotres}[1]{{{\color{blue} 
}}}
\newcommand{\verluegocuatro}[1]{{\color{black} 
}}
\newcommand{\verluegocinco}[1]{{\color{blue-green} 
}}
\newcommand{\verluegoalphados}[1]{{\color{some-blue} 
}}

\newtheorem{theorem}{Theorem}
[section]
\newtheorem{lemma}[theorem]{Lemma}

\newtheorem{prop}[theorem]{Proposition}

\theoremstyle{definition}

\newtheorem{remark}[theorem]{Remark}

\usepackage[
backend=biber,
giveninits=true, % CAREFUL this will force initials (inits) instead of full names on authors.
style=numeric, %the standard? There seems to be no standard…
]{biblatex}

\DeclareFieldFormat[article]{author}{\textsc{#1}}
\DeclareFieldFormat[article]{title}{\textit{#1}}
\DeclareFieldFormat[article]{journaltitle}{{#1}}
\DeclareFieldFormat[article]{volume}{{#1}}
\DeclareFieldFormat[article]{number}{{ no. #1}}
\DeclareFieldFormat[article]{pages}{#1}
\DeclareFieldFormat[article]{year}{(#1)}

\DeclareFieldFormat[book]{title}{{#1}}
\DeclareFieldFormat[book]{publisher}{#1}

\DeclareFieldFormat[online]{title}{\textit{#1}}
\DeclareFieldFormat[online]{date}{(#1)} % Works!
\DeclareFieldFormat[online]{url}{\url{#1}}

\DeclareFieldFormat[incollection]{title}{#1}

\AtEveryBibitem{%
	\clearfield{month}%
	\clearfield{day}%
	\clearfield{urldate}%
	\clearfield{doi} % So that DOIs show up as Springer/JGA wants
	\clearfield{eprintclass}
	\clearfield{url}
	\clearfield{pubstate}
	\clearfield{editor}
	\clearfield{isbn}
	\clearfield{issn}
	\clearfield{journaltitle}  % Careful. Comment only when you want shortjournal names.
}

\renewbibmacro*{journal}{%
	\ifboolexpr{
		test {\iffieldundef{shortjournal}}
		and
		test {\iffieldundef{journaltitle}}
	}
	{}
	{\printtext[journaltitle]{%
			\printfield[titlecase]{shortjournal}%
			\setunit{\subtitlepunct}%
			\printfield[titlecase]{journaltitle}}}
}

\renewbibmacro*{url+urldate}{%
	\printfield{url}%
}

\renewbibmacro{in:}{} % This gets rid of 'in' within the phrase 'in proceedings of the ams' or whatever 

\date{}
\title{\sc Empty title}
\author{\sc Javier Minguillón, Fernando Soria, Ana Vargas}

\title{On the Pointwise Convergence of Solutions to the Schrödinger Equation Along Certain Highly Tangential Curves}

\date{\today}
\begin{document}
	\maketitle
	\footnote{The authors are partially supported by Grant
		PID2022-142202NB-I00 funded by AEI/10.13039/501100011033.}

\section{Introduction and reductions}\label{section1_introduction and reductions}
\begin{abstract}
We investigate the Sobolev regularity required for almost everywhere convergence to the initial datum of solutions to the linear Schrödinger equation along certain tangential curves. In the regime $\alpha<\tfrac12$, we analyze maximal estimates for expressions of the form $e^{it\Delta}f(x+\gamma(t))$ over specific $\alpha$-Hölder curves $\gamma$ and initial data $f\in H^s(\mathbb{R}^n)$. For the model family $\gamma(t)=(t^{\alpha_1},\ldots,t^{\alpha_n})$, where $\alpha=\min_j \alpha_j$, we show that the critical regularity is $s=\max\left\{\frac{1-2\alpha}{2},\frac{n}{2(n+1)}\right\}.$
\end{abstract}

Consider the linear Schrödinger equation on $ \R^n\times \R $, $ n\geq 1, $ given by
\begin{equation}
	\label{equation_schrodinger with boundary datum}
	\begin{cases}
		i\partial_t u(x,t) - \Delta_x u (x,t) = 0,
		\\
		u(x,0) = f(x).
	\end{cases}
\end{equation}
Its solution can be formally expressed as
\begin{equation}
	\schtf{x} = \int_{\R^n} e^{2\pi i x\cdot \xi} e^{4\pi^2 it|\xi|^2} \widehat f(\xi) d\xi.
\end{equation}

The question of pointwise convergence was first proposed by Carleson in 1980 \cite{Carleson1980AnalyticProblemsRelated}. He asked for the values of $ s>0 $ for which
\begin{equation}\label{equation_classical ae convergence}
	\lim_{t\to 0} \schtf{x} = f(x), \quad \text{a.e.}\quad x\in \R^n,
\end{equation}
holds true for all functions $ f\in H^s(\R^n) $. Carleson \cite{Carleson1980AnalyticProblemsRelated} proved this convergence when $ n=1 $ and $ s\geq \frac14$. Shortly after, Dahlberg and Kenig \cite{DahlbergKenig2006NoteAlmostEverywhere} showed that this condition was necessary, as \eqref{equation_classical ae convergence} fails whenever $ s<\frac14 $.

Over the years, many researchers have contributed to this problem, including Carbery, Cowling, Sjölin, Vega, Moyua, Vargas, Tao, and Bourgain, among others. A major breakthrough came in higher dimensions. In 2016, Bourgain \cite{Bourgain2016NoteSchrodingerMaximala} proved the necessity of the condition $ s\geq \frac{n}{2(n+1)} $. On the sufficiency side, Du, Guth and Li \cite{DuGuthLi2017SharpSchrodingerMaximal} solved the case $ n=2 $ with $ s>\frac13 $ in 2017, and in 2019, Du and Zhang \cite{DuZhang2019Sharp$L^2$Estimates} proved the sharp result for general $ n\geq 3 $ with $ s>\frac{n}{2(n+1)} $.

We consider a variation of this question by approaching the initial datum along a curve. Fix $0<\alpha\leq 1$ and $\calpha\geq1.$ Define the family of curves
\begin{equation}
	\Gamma_{\calpha}^\alpha
	:=
	\left\lbrace
	\gamma:[0,1] \to \mathbb R^n
	: \text{for all } t,t'\in [0,1],\; |\gamma(t)-\gamma(t')| \leq \calpha|t-t'|^\alpha
	\right\rbrace.
\end{equation}
The problem of convergence along such curves was first studied by Cho, Lee, and Vargas \cite{ChoLeeVargas2012ProblemsPointwiseConvergence}. In dimension $ n=1 $, they proved that for $\alpha$-Hölder curves with $\alpha\in (0,1]$, the limit $\lim_{t\to 0} e^{it\Delta}f(x+\gamma(t)) = f(x)$ holds a.e. for $ f\in H^s(\R) $ with $ s>\max\left\lbrace \frac12-\alpha, \frac14 \right\rbrace $, which they also showed to be sharp up to the endpoint. Later, in 2021, Li and Wang \cite{WangLi2026ConvergencePropertiesGeneralized} extended this result to dimension $ n=2 $ for the range $ \frac12\leq \alpha\leq 1 $, proving convergence for $ s > \frac38 $. More recently, in 2023, Cao and Miao \cite{CaoMiao2023SharpPointwiseConvergence} gave a proof for general dimension $ n $, for curves with Hölder index $ \frac12\leq \alpha\leq 1 $, and for $ s>\frac{n}{2(n+1)} $. Another simpler proof of this fact was given by the first author in \cite{Minguillon2024NoteAlmostEverywherea}.

In this paper, we are interested in the more delicate regime $\alpha < \frac12$, where the curves become highly tangential. Specifically, we ask what sufficient conditions on the regularity of $f$ ensure that
\begin{equation}\label{eq_ae convergence along a curve}
	\lim_{t\to 0} \schtf{x+\gamma(t)} = f(x), \quad \text{a.e.}\quad x\in \R^n.
\end{equation}
\begin{remark}
	A change of variables shows that it is enough to consider the case $ \calpha = 1 $. From now on we assume $ \calpha = 1 $ 
	and will denote $\Gamma_1^\alpha$ by $ \Gamma^\alpha.$
\end{remark}
A maximal bound is a sufficient condition to guarantee the above convergence. Let $ B_r^n(x_0) $ denote the ball of radius $ r>0  $ centered at $ x_0\in\R^n $.
Assume that there exists a positive constant $C$ such that, for a $ \gamma\in\Gamma^\alpha $,
\begin{equation}\label{equation_maximal estimate non-reduced}
	\left\|
	\sup _{0<t<1}\left|
	\schtf{x + \gamma(t)}
	\right|
	\right\|_{L^2\left(B^{n}_1(0)\right)}
	\leq
	C
	\|f\|_{
		H^{r}
		\left(\R^n\right)}
\end{equation}
holds for all
$ f $
$ \in H^{r}(\R^n).$ 
Then, \eqref{eq_ae convergence along a curve} holds for every $f\in H^s,$ and all $\esnalpha>r$.

There is a well known estimate.
\begin{remark}
	Estimate \eqref{equation_maximal estimate non-reduced} holds for any $r>\frac n2$.
\end{remark}
We are going to reduce \eqref{equation_maximal estimate non-reduced} to the bound below. We proceed as in \cite{DuZhang2019Sharp$L^2$Estimates}.
%We begin with a definition.
%
%\begin{defin}
%	For fixed $0<\alpha\leq 1$ and $ R>1 $, we define
%	\begin{equation}
%		\GammaR{R^{-1}}
%		:=
%		\left\lbrace
%		\gamma:[0,R^{-1}] \to \mathbb R^n
%		: \text{for all } t,t'\in [0,R^{-1}],\; |\gamma(t)-\gamma(t')| \leq |t-t'|^\alpha
%		\right\rbrace.
%	\end{equation}
%\end{defin}
By Littlewood-Paley decomposition, the time localization lemma (e.g. Lemma 3.1 in S. Lee \cite{Lee2006PointwiseConvergenceSolutions}) and parabolic rescaling, \eqref{equation_maximal estimate non-reduced} can be reduced to the following estimate.
Assume that there exists $\esnalpha>0$ and a constant $C_{\esnalpha}$ satisfying that, for a $ \gamma\in \Gamma^\alpha $,
\begin{equation}
	\left\|
	\sup _{0<t < R}\left|
	e^{i t \Delta} f
	\left(
	x + R\gamma\left(\frac{t}{R^2}\right)
	\right)
	\right|
	\right\|_{L^2\left(B^n_R(0)\right)}
	\leq
	C_{\esnalpha} R^{\esnalpha}
	\|f\|_2
	\label{eq_theorem_Theorem 1.3. from Du Zhang modified}
\end{equation}
holds for all $R \geq 1$ and all $f$ with $\operatorname{supp} \widehat{f} \subset A(1)=\left\{\xi \in \mathbb{R}^n:|\xi| \sim 1\right\}$. Then, \eqref{equation_maximal estimate non-reduced} holds for any $r>\esnalpha$ and all $f\in H^r$.

As shown by counterexamples found in \cite{Bourgain2016NoteSchrodingerMaximala} and \cite{ChoLeeVargas2012ProblemsPointwiseConvergence},
\begin{equation}
	s>\max\left\lbrace {\frac{1-2\alpha}{2}}, {\frac{n}{2(n+1)}}\right\rbrace
\end{equation}
is a necessary condition for \eqref{eq_theorem_Theorem 1.3. from Du Zhang modified} to hold.

We are interested in finding the infimum of the values $\esnalpha$ that satisfy
\eqref{eq_theorem_Theorem 1.3. from Du Zhang modified}.
In the following section, we discuss a theorem that provides the critical value for $\esnalpha$ when we restrict to a certain family of curves.

\section{Main results}
	Let $R>>1$. 
	Let $\alpha_j>0$ for $j=1,...,n$. Denote 
%	$\alpha = \max_j\left\lbrace \alpha_j\right\rbrace.$
	$\alpha = \min_j\left\lbrace \alpha_j\right\rbrace.$
	Further suppose that $\alpha<\frac12$. Define the model curve
	\begin{equation}
		\gamma_0(t) = (t^{\alpha_1},...,t^{\alpha_n}).
	\end{equation}
	for $t\in[0,1]$. Our main result is the following theorem for this curve.
	\begin{theorem}\label{theorem_a.e. convergence for specific curve}
		Let 
		\begin{equation}
			s_0(n,\alpha) = \max\left\lbrace {\frac{1-2\alpha}{2}}, {\frac{n}{2(n+1)}}\right\rbrace =\begin{cases}\frac{1-2\alpha}{2} \text{, if } \alpha< \frac1{2(n+1)},\\ \frac{n}{2(n+1)} \text{, if } \alpha\geq \frac1{2(n+1)}.\end{cases}
		\end{equation}
		Then,
		\begin{equation}\label{equation_theta ae convergence}
			\lim_{t\to 0} \schtf{x + \gamma_0(t)} = f(x), \quad \text{a.e.}\quad x\in \R^n,
		\end{equation}
		for all $f\in H^{s}(\R^n)$ and $s>s_0(n,\alpha)$.
	\end{theorem}
	Denote $\theta(t) \defeq \theta_R(t) \defeq R\gamma_0\left(\frac{t}{R^2}\right)$ for $0\leq t\leq R$. That is, consider
	\begin{equation}\label{eq_definition of theta}
		\theta(t) 
		= 
		\left(
			R^{1-2\alpha_1}t^{\alpha_1}, ... , R^{1-2\alpha_n}t^{\alpha_n}
		\right).
	\end{equation}
	
	By the considerations in Section \ref{section1_introduction and reductions}, the following theorem implies the above \Cref{theorem_a.e. convergence for specific curve}.

	\begin{theorem}\label{theorem_maximal bound that implies ae convergence}
		Let $n\geq 1$ and $0<\alpha<\frac12$. For any $\varepsilon>0$, there exists a positive constant $C_\varepsilon$ such that the following holds. For all  $f:\R^n\to\C$ such that 
		%   $\supp\left(\widehat f\right)$, 
		$\supp\widehat f\subset B^n_1(0)$,
		%	$\suppp{\widehat f}$,
		and all $R\geq 1$,
		\begin{align}
			\left\|
			\sup_{0<t<R} 			
			\left|\schtf{x+\theta(t)}\right|
			\right\|
			_{L^2\left(B_R(0),dx\right)} 
			&\leq
			C_\varepsilon R^{s_0(n,\alpha)+\varepsilon} \|f\|_2,
			\label{eq_izquierda - constante peor posible v2}
		\end{align}
		where $s_0(n,\alpha) = \max\left\lbrace {\frac{1-2\alpha}{2}}, {\frac{n}{2(n+1)}}\right\rbrace 
		$. 	
	\end{theorem}
	
	We are going to use the following result from \cite{DuZhang2019Sharp$L^2$Estimates}. We denote $B_r^{n+1}(x',t')$ the $(n+1)-$dimensional ball of radius r centered at $(x',t')$ where $x'\in \R^n$ and $t'\in \R$. 
	
	\begin{theorem}[Corollary 1.7 in \cite{DuZhang2019Sharp$L^2$Estimates}]\label{theorem_Cor1.7 of Du Zhang}
		Let $n \geq 1$. For any $\varepsilon>0$, there exists a constant $C_{\varepsilon}$ such that the following holds, for all $R \geq 1$ and all $f$ with $\operatorname{supp} \widehat{f} \subset B_1^n(0)$. Fix $L\in \N$. Suppose that $X=\bigcup_{k=1}^L \tildeq_k$ is a union of lattice cubes in $B_R^{n+1}(0)$. Let $1 \leq \beta \leq n+1$ and
		
		\begin{equation}\label{equation_definition of beta density of X}
			\phi
			= \phi_{X,\beta, R}
%			= \phi_{\beta,X}
			=
			\max _{\substack{ B_r^{n+1}(x',t') \subset B_R^{n+1}(0) \\ (x^{\prime},t^{\prime}) \in \mathbb{R}^{n+1}, r \geq 1}} 
			\frac{\#\left\{\tildeq_k\subset X: \tildeq_k \subset B^{n+1}_r\left(x^{\prime},t^{\prime} \right)\right\}}{r^\beta}.
		\end{equation}
		Then,
		\begin{equation}
			\left\|e^{i t \Delta} f\right\|_{L^2(X,\dxt)} 
			%\left\|\schtf{x}\right\|_{L^2(X,\dxt)} 
			\leq C_{\varepsilon} \phi^{\frac{1}{n+1}} R^{\frac{\beta}{2(n+1)}+\varepsilon}\|f\|_2.
		\end{equation}
	\end{theorem}
	
	We also need a variant of a well known result. 
	\begin{theorem}\label{theorem_ball-to-band}
		Fix $R>>1$, $s>0$, and $X \subset \R^n\times [0,R]$. Suppose there exists $C>1$ 
		%		independent of $R$, 
		%unnecessary assumption. True in the context of our purposes but whatever.
		such that, for all $f\in L^2$ such that $\supp{\widehat{f}}\subset B_1^n(0)$ and all $k\in R\Z^n$, 
		\begin{align}
			\left\| \schtf{y} \right\|_{
				L^2\left(
				\left(B_R^n(k) \times [0,R]\right)\cap X, dydt\right) }
			&\leq C R^s \|f\|_2.
			\label{eq_bound-ball}
			\intertext{Then, for every $\epsilon>0$ there exists $C_{n,\epsilon}>0$ only depending on the dimension such that for all $f\in L^2(\R^n)$ satisfying that $\supp{\widehat{f}}\subset B_1^n(0)$,}
			\left\| \schtf{y} \right\|_{
				L^2\left(
				\left(\R^n \times [0,R]\right)\cap X, dydt\right) }
			&\leq C_{n,\varepsilon} C R^{s+\varepsilon} \|f\|_2.
			\label{eq_bound-band}
		\end{align}
	\end{theorem}
%\verluego{Creo que podemos prescindir por entero de C_n. Ver último paso de la prueba. De la última subsección}
The proof is similar to that of the original result. See Lemma 8 in \cite{Rogers2008LocalSmoothingEstimate}.

\section{Maximal bound over the model curve: Proof of Theorem \ref{theorem_maximal bound that implies ae convergence}.
%	 with Fourier support property better justified
}
In this section, we prove \Cref{theorem_maximal bound that implies ae convergence}. It follows from the following proposition.
\begin{prop}
	Let $n\geq 1$ and $0<\alpha<\frac12$. For any $\varepsilon>0$, there exist positive constants $C$ and $C_\varepsilon$ and such that the following holds. For all  $f:\R^n\to\C$ such that 
	%   $\supp\left(\widehat f\right)$, 
	$\supp\widehat f\subset B^n_1(0)$,
	%	$\suppp{\widehat f}$,
	and all $R\geq 1$,
	\begin{align}
		\left\|
		\sup_{0<t<1} 			
		\left|\schtf{x+\theta(t)}\right|
		\right\|
		_{L^2\left(B_R(0),dx\right)} 
		&\leq
		C R^{\frac{1-2\alpha}{2}} \|f\|_2,
		\label{eq_izquierda t<1}
		\intertext{and}
		\left\|
		\sup_{1<t<R} 			
		\left|\schtf{x+\theta(t)}\right|
		\right\|
		_{L^2\left(B_R(0),dx\right)} 
		&\leq
		C_\varepsilon R^{s_0(n,\alpha)+\varepsilon} \|f\|_2,
		\label{eq_izquierda t>1}
	\end{align}
	where $s_0(n,\alpha) = \max\left\lbrace {\frac{1-2\alpha}{2}}, {\frac{n}{2(n+1)}}\right\rbrace $
\end{prop}

We first prove \eqref{eq_izquierda t<1}. Before this, we need to introduce the following stability property of the Schrödinger operator. More general versions of the estimate below appear in an article of T. Tao \cite{Tao1999BochnerRieszConjectureImplies} from 1999 and an article of M. Christ \cite{Christ1988RegularityInversesSingular} from 1988.

\begin{prop}\label{proposition_locally constant with series}
Suppose that 
$ \supp\widehat f \subset B^n_1(0)$. 
%$ \supp{\widehat f} \subset B_1(\xi_0) $
%	for some $ \xi_0\in \R^n. $
If $ |x-y| \leq 8 $ and $ |t-s|\leq 8, $ then,
\begin{equation}
	\left|e^{it\Delta}f(x)\right|
	\leq
	\sum_{\mathfrak{l}\in\Z^n} 
	\frac{1}{(1+|\mathfrak{l}|)^{n+1}} \left|e^{is\Delta}f_\mathfrak{l}(y)\right|,
\end{equation}
where $ \widehat{f_\mathfrak{l}}(\xi) = e^{2\pi i\mathfrak{l}\xi}\widehat f(\xi)$.
Hence, by Hölder's inequality, for $p\geq1$,
\begin{equation}\label{equation_stability lemma pointwise}
	\left|e^{it\Delta}f(x)\right|^p
	\leq
	\sum_{\mathfrak{l}\in\Z^n} \frac1{(1+|\mathfrak{l}|)^{n+1}}
	\left|e^{is\Delta}f_\mathfrak{l}(y)\right|^p.
\end{equation}
Therefore, if $ |x-x'| \leq 4 $ and $ |t-t'|\leq 4 $, then, for all $p>0$
\begin{equation}\label{equation_stability lemma aka locally constant property substitute}
	\left|e^{it\Delta}f(x)\right| ^p
	\leq
	\sum_{\mathfrak{l}\in\Z^n} 
	\frac{1}{(1+|\mathfrak{l}|)^{n+1}} \int_{t'}^{t'+1}\int_{B_1(x')}\left|e^{is\Delta}f_\mathfrak{l}(y)\right|^p dyds,
\end{equation}	
and,
\begin{equation}
	\int_t^{t+1}\int_{B_1(x)} \left|e^{ir\Delta}f(u)\right|^pdudr
	\leq
	\sum_{\mathfrak{l}\in\Z^n} 	\frac{1}{(1+|\mathfrak{l}|)^{n+1}} \int_{t'}^{t'+1}\int_{B_1(x')}\left|e^{is\Delta}f_\mathfrak{l}(y)\right|^p dyds.
\end{equation}
\end{prop}

We turn to the proof of \eqref{eq_izquierda t<1}.
\begin{proof} 
	Fix $f$ with $\supp \widehat f \subset B^n_1(0)$ and $x\in\R^n$. Define $N(\theta)$ to be the set of integers inside the $1$-neighborhood of $\theta([0,1])$. The integrand on the left hand side of \eqref{eq_izquierda t<1} is
\begin{align}
		\sup_{0<t<1} \left|e^{it\Delta}f(x+\theta(t))\right| 
		&= 
		\max_{k\in N(\theta)} \sup_{t\in \theta^{-1}(B_1(k))\cap[0,1]} 	\left|e^{it\Delta}f(x+\theta(t))\right|.
	\intertext{For each $k$, choose $t_k$ as an element in $\theta^{-1}(B_1(k))\cap[0,1]$ such that
	}
		&\leq 2 \max_{k\in N(\theta)} \left|e^{it_k\Delta}f(x+\theta(t_k))\right|.
	\intertext{Since $|t_k|\leq 1$ and $|x+\theta(t_k)-(x+k)|\leq1$, by 
%		\eqref{equation_stability lemma aka locally constant property substitute} 
		Proposition \ref{proposition_locally constant with series}
		we obtain
	}	
%		&\leq 2 \max_{k\in N(\theta)}
%		\sum_{\mathfrak l\in\Z^n} \frac{1}{(1+|\mathfrak l|)^{n+1}}
%		\int_0^1\int_{B_1(x+k)} |e^{is\Delta}f_{\mathfrak l}(y)| dyds.
%	\intertext{By Cauchy-Schwartz inequality, this becomes}
	&\lesssim \max_{k\in N(\theta)}
	\left(
	\sum_{\mathfrak l\in\Z^n} \frac{1}{(1+|\mathfrak l|)^{n+1}} 
	\int_0^1\int_{B_1(x+k)}|e^{is\Delta}f_\mathfrak{l}(y) |^2 dyds
	\right)^\frac12.
\end{align}

Now take the squared $2-$norm of the above,
\begin{align}
	& \left\|\sup_{0<t<1}\left|e^{it\Delta}f(x+\theta(t))\right|\right\|_{L^2(B_R(0),dx)}^2
	\\
	\lesssim &
	\int_{B_R(0)} \left( 
	\max_{k\in N(\theta)}
	\sum_{\mathfrak l\in\Z^n} \frac{1}{(1+|\mathfrak l|)^{n+1}} 
	\int_0^1\int_{B_1(x+k)} 
	|e^{is\Delta}f_\mathfrak{l}(y) |^2 
	dyds 
	\right)dx.
	\intertext{Majorize the maximum over $k$ by a sum over $k$ and apply Fubini's theorem to get}
	\lesssim &
	\sum_{\mathfrak l\in\Z^n} \frac{1}{(1+|\mathfrak l|)^{n+1}} 
	\int_0^1\int_{\R^n} 
	|e^{is\Delta}f_\mathfrak{l}(y)|^2 
	\left( 
		\sum_{k\in N(\theta)}
		\int_{B_R(0)} \chi_{B_1(x+k)}(y)dx
	\right)
	dyds .
	\intertext{The innermost integral is equal to a dimensional constant, and the amount of terms in the sum over $k$ is about $\len( \theta[0,1])\simeq R^{1-2\alpha}$, as calculated later in \eqref{eq_bound length of curve}. Therefore the above is}
	\simeq &
	R^{1-2\alpha} \sum_{\mathfrak l\in\Z^n} \frac{1}{(1+|\mathfrak l|)^{n+1}} 
	\int_0^1\left\|e^{is\Delta}f_{\mathfrak l}(y)\right\|_{L^2(dy)}^2 ds
	\\
	= &
	R^{1-2\alpha} \sum_{\mathfrak l\in\Z^n} \frac{1}{(1+|\mathfrak l|)^{n+1}}
	\|f_{\mathfrak l}\|_2^2
	\\
	\simeq &  R^{1-2\alpha} \|f\|_2^2.
\end{align}
\end{proof}

%\subsection{Proof of $s_0(n,\alpha)\geq \max\left\lbrace\frac{1-2\alpha}{2}, \frac{n}{2(n+1)}\right\rbrace$.}
We can now move on to completing the proof of \eqref{eq_izquierda t>1}. Before that, we need a few definitions.

Fix $f$ with $\supp\widehat f\subset B^n_1(0)$. Fix $j\in \Z^n$, $|j|\leq R$. 
Pick $t_j\in[1,R]$ such that
\begin{equation}\label{eq_lcp consequence for Qx}
	\sup\limits_{1<t<R}
	|\schtf{j+\theta(t)}|
%	\simeq
	\leq 2
	%		\sup\limits_{\lambda<t<2\lambda}
		\left|e^{it_j\Delta}f(j+\theta(t_j))\right| 
	,
\end{equation}
Given $j\in\Z^n$, denote by $\tildeq_j$ the $n+1$-dimensional lattice cube that contains $(j+\theta(t_j),t_j)$.
Define 
\begin{align}
	\verluegocinco{X_f=} X = \left( B_R^{n}(0)\times[1,R] \right) \cap
	\left(
	\bigcup\limits_{|j|\leq R} \tildeq_j
	\right),
	\intertext{and, for each dyadic number $1\leq \lambda\leq R$, define}
	X_{\lambda} = X\cap(\R^n\times[\lambda,2\lambda]).
\end{align}
Observe that, by definition \eqref{eq_definition of theta} of $\theta$,
we have 
%that $X\subset B_{2R}^{n+1}(0)$ and 
that $\{j: \tildeq_j \subset X\}\subset B_{2R}^n(0).$

Let $\eta\geq 1$ be a dyadic number. 
Let $\mathcal{F}_\eta$ be the collection of all $(n+1)$-dimensional lattice cubes $\tildeq \subset \mathbb{R}^{n+1}$ satisfying
\begin{equation}\label{eq_definition of overlap in the family of X lambda eta}
\#\{y \in B_R^n(0) \cap \mathbb{Z}^n : \tildeq  = \tildeq_y\} \sim \eta.
\end{equation}

Also define
\begin{equation}
	X_{\lambda,\eta} 
	\defeq
	X_{\lambda} \cap \left(\bigcup\limits_{\tildeq \in \mathcal F_\eta}\tildeq\right).
\end{equation}

\begin{remark}\label{remark_on the size of eta}
	We necessarily have
%	$1\leq \eta \lesssim R^{1-2\alpha}\lambda^{\alpha-1}.$ 
	$1\leq \eta \lesssim R^{1-2\alpha}\lambda^{\alpha-1}.$ 
	Otherwise, 
	$X_{\lambda,\eta}=\emptyset$.
%	$\mathcal F_\eta=\emptyset$
\end{remark}
\begin{proof}
		Fix $\lambda\geq 1$ and $\eta\geq 1$. Fix a lattice cube $\tildeq \in \mathcal F_\eta$ such that $\tildeq \subset \R^n\times[\lambda,2\lambda]$. Let $\lambda\leq\lambda_1\leq 2\lambda-1$ be the integer such that 
		%	$\tildeq\subset \R^n\times[\lambda_1,2\lambda_1]$. 
		$\tildeq\subset \R^n\times[\lambda_1,\lambda_1+1]$. 
		Consider the set
		\begin{equation}
			\left\{ j\in\Z^n: \tildeq = \tildeq_j \right\} = \left\{j\in\Z^n : (j+\theta(t_j),t_j)\in\tildeq\right\}
			\label{eq_set of js that make ~Qj to be ~Q}
		\end{equation}
		The set is nonempty by definition of $\eta$. We will bound the cardinal of this set from above.  
		Fix an element $j$ in the set $\eqref{eq_set of js that make ~Qj to be ~Q}$. Take $k\in \Z^n$. Call $c_n=\sqrt{n+1}$. In this setting, in order to have
			$(k+\theta(t_k),t_k)\in\tildeq$ it is necessary to satisfy
		\begin{align}
			|j+\theta(t_j) - (k + \theta(t_k))|&\leq c_n.
		\intertext{
		For this, it is necessary to satisfy
		}
			|j+\theta(t_j) - (k + \theta(t))|&\leq c_n \text{ for some } t\in [\lambda_1,\lambda_1+1].
		\end{align}
		This is equivalent to having that $k$ lies in the tubular $c_n$-neighborhood of
		$\{j+\theta(t_j)- \theta(t):\, t\in[\lambda_1,\lambda_1+1] \}$.
		This tubular $c_n$-neighborhood has an $n-$dimensional volume bounded above by
		\begin{equation}
			\widetilde{c_n}\cdot\max\{1,\len\left(\theta[\lambda_1, \lambda_1+1]\right)\}.
		\end{equation}
		
	We conclude that
	$
		\#\{k\in\Z^n:\tildeq = Q_k\}\leq \widetilde{c_n} \cdot \max\{1,\len\left(\theta[\lambda_1, \lambda_1+1]\right)\}.
%		\label{eq_the set of js is bounded above by the length of theta}
	$
	Now,
	\begin{align}
		%		\len\left(\theta[t_{\tildeq}, t_{\tildeq}+1]\right)
		\len\left(\theta[\lambda_1, \lambda_1+1]\right)
		&= \int_{\lambda_1}^{\lambda_1+1} |\theta'(s)|ds 
		\lesssim \int_{\lambda_1}^{\lambda_1+1}R^{1-2\alpha}s^{\alpha-1} ds
		\\
		&\leq 
		\int_{\lambda_1}^{\lambda_1+1}R^{1-2\alpha}\lambda^{\alpha-1} ds
		= R^{1-2\alpha}\lambda^{\alpha-1} .
		\label{eq_bound length of curve}
	\end{align}
	By definition of $\eta$, 
%	and the bounds in \eqref{eq_the set of js is bounded above by the length of theta} and \eqref{eq_bound length of curve}, 
	we are done.

\end{proof}
%	Observe that, whenever $\lambda\geq R^{\frac{1-2\alpha}{1-\alpha}}$, it can only be $\eta \sim 1$.

Let us perform two reductions to prove \eqref{eq_izquierda t>1}.
\begin{lemma}\label{lemma_reduce integration support of central estimate}
	%		By definition of $X_{\lambda,\eta}$, in
%	In order to prove \eqref{eq_izquierda t>1}, it suffices to prove that, f
	For all $\varepsilon>0$, there exists $C_\varepsilon>0$ such that
	\begin{equation}\label{eq_lemma_reduce integration support of central estimate}
		\left\|\schtf{y}\right\|_{L^2(X_{\lambda,\eta},dydt)}
		\lessssim 
		\eta^{-\frac12} R^{s_0(n,\alpha)+\epsilon} \|f\|_2,
	\end{equation}
	for all 
%	$1\leq \lambda\leq R^{\frac{1-2\alpha}{1-\alpha}}$, 
	$1\leq \lambda\leq R$, 
	all $1\leq\eta\lesssim R^{1-2\alpha}\lambda^{\alpha-1}$.
%	and 
%	$k\in \lambda\Z^n$.
\end{lemma}
We delay the proof of the lemma for the moment. 
\begin{remark}\label{remark_reduce integration support}
	Estimate \eqref{eq_izquierda t>1} is a consequence of Lemma \ref{lemma_reduce integration support of central estimate}.
\end{remark}

\begin{proof}
	Write
	\begin{align}
		&\phantom{=} 
		\left\|
		\sup_{1\leq t<R} 			
		\left|\schtf{x+\theta(t)}\right|
		\right\|
		_{L^2\left(B_R(0),dx\right)} ^ 2
		\leq
		\sum_{\substack{j\in\Z^n\\ |j|\leq R}} \left\|    
		\sup_{1< t<R} |e^{it\Delta}f(x+\theta(t))|
		\right\|_{L^2\left(B_1(j),dx\right)} ^ 2 .
	\end{align}
		For each $j\in\Z^n$ with $|j|\leq R$, 
		we have chosen $t_j$ and $\lambda$ satisfying
	$t_j\in[\lambda,2\lambda] $ such that 
	\begin{align}
		\label{eq_lcp consequence for Qx 2nd instance}
		\sup\limits_{1<t<R}
		|\schtf{j+\theta(t)}|
		\leq 2
		\left|e^{it_j\Delta}f(j+\theta(t_j))\right| 
		,
	\end{align}
	
	By Proposition \ref{proposition_locally constant with series}, denoting  $\widehat{f_{\mathfrak{l}} }
%	(\xi)
%	= (f_{\mathfrak{k}})_\mathfrak{l}$, 
	= e^{2\pi i\mathfrak{l}\xi}\widehat f(\xi)$,
	we deduce that
	\begin{align}
		&
		\phantom{=} 
		\left\|
		\sup_{1< t<R} 			
		\left|\schtf{x+\theta(t)}\right|
		\right\|
		_{L^2\left(B_R(0),dx\right)} ^ 2
		\\
		& 
		\leq
		\sum_{\substack{j\in\Z^n\\ |j|\leq R}}
		\left(
		\sum_{\mathfrak{k}\in\Z^n} \frac{1}{(1+|\mathfrak{k}|)^{n+1}}
		%			\left\|
		\sup_{1< t<R}
		|e^{it\Delta}f_{\mathfrak{k}}(j+\theta(t))|^2
		%			\right\|_{L^2\left(B_1(j),dx\right)}
		%			Esta norma en x sobre una bola unidad desaparece tal cual al cambiar x por j, claro
		\right),
		\intertext{and, by \eqref{eq_lcp consequence for Qx 2nd instance}, this is}
		&\lesssim
		\sum_{\substack{j\in\Z^n\\ |j|\leq R}}
		\left(
		\sum_{\mathfrak{k}\in\Z^n} \frac{1}{(1+|\mathfrak{k}|)^{n+1}}
		%			\left\|
		%		\sup_{1\leq t<R}
		|e^{it_j\Delta}f_{\mathfrak{k}}(j+\theta(t_j))|^2
		\right).
		\end{align}
	By the above inequality and Proposition \ref{proposition_locally constant with series}, denoting  $\widehat{f_{\mathfrak{k},\mathfrak{l}} }
	%	(\xi)
	%	= (f_{\mathfrak{k}})_\mathfrak{l}$, 
	= e^{2\pi i\mathfrak{l}\xi}e^{2\pi i\mathfrak{k}\xi}\widehat f(\xi)$,
	we deduce that
	\begin{align}
		&\phantom{=} 
		\left\|
		\sup_{1< t<R} 			
		\left|\schtf{x+\theta(t)}\right|
		\right\|
		_{L^2\left(B_R(0),dx\right)} ^ 2
		\\
		& 
		%		\color{blue}
		\lesssim
		\sum_{\substack{j\in\Z^n\\ |j|\leq R}}
		\left(
		\sum_{\mathfrak{l},\mathfrak{k}\in\Z^n} 
		\frac{1}{[(1+\mathfrak{k})(1+|\mathfrak{l}|)]^{n+1}} \int_{t_j}^{t_j+1}\int_{B_1(j)}\left|e^{is\Delta}f_{\mathfrak{k},\mathfrak{l}}(z+\theta(t_j))\right|^2dzds
		\right)
		,
		\intertext{Change $y=z-\theta(t_j)$ in the inner integral, and it becomes}
		& 
		=
		\sum_{\substack{j\in\Z^n\\ |j|\leq R}}
		\left(
		\sum_{\mathfrak{l},\mathfrak{k}\in\Z^n} 
		\frac{1}{[(1+\mathfrak{k})(1+|\mathfrak{l}|)]^{n+1}} \int_{t_j}^{t_j+1}\int_{B_1(j +\theta(t_j))}\left|e^{is\Delta}f_{\mathfrak{k},\mathfrak{l}}(y)\right|^2 dyds
		\right).
		\intertext{By definition of $\tildeq_j$, the above sum is}
		&\lesssim
		\sum_{\substack{j\in\Z^n\\ |j|\leq R}} 
		\left(
		\sum_{\mathfrak{l},\mathfrak{k}\in\Z^n} 
		\frac{1}{[(1+\mathfrak{k})(1+|\mathfrak{l}|)]^{n+1}}
		\int_{ 
			\tildeq_j
		} 
		\left|e^{is\Delta}f_{\mathfrak{k},\mathfrak{l}}(y)\right|^2 dyds
		\right).
%		\intertext{and, by Cauchy-Schwarz, it becomes}
%		&\lesssim
%		\sum_{\substack{j\in\Z^n\\ |j|\leq R}} 
%		\sum_{\mathfrak{l},\mathfrak{k}\in\Z^n} 
%		\frac{1}{[(1+\mathfrak{k})(1+|\mathfrak{l}|)]^{n+1}}
%		\int_{ 
%			\tildeq_j
%		} 
%		\left|e^{is\Delta}f_{\mathfrak{k},\mathfrak{l}}(y)\right|^2 dyds,
		\intertext{Now, for each $i$, choose $\lambda,\eta$ such that $Q_i\subset X_{\lambda,\eta}$. Then, by Fubini's theorem, the definition of $X_{\lambda,\eta}$, and Remark \ref{remark_on the size of eta},} 
		&\lesssim 
		\sum_{\mathfrak{l},\mathfrak{k}\in\Z^n} 
		\frac{1}{[(1+\mathfrak{k})(1+|\mathfrak{l}|)]^{n+1}}
		\sum_{\substack{\lambda\sim 1}}^{R}
		\sum_{\substack{\eta\sim 1}}^{R^{1-2\alpha}\lambda^{\alpha-1}}
		\eta \int_{X_{\lambda,\eta}} |e^{is\Delta}f_{\mathfrak{k},\mathfrak{l}} (y)|^2 dyds.
	\end{align}
	Notice that the innermost sum has only the term for $\eta=1$ whenever $\lambda \geq R^{\frac{1-2\alpha}{1-\alpha}}$.
	We can apply Lemma \ref{lemma_reduce integration support of central estimate} and obtain \eqref{eq_izquierda t>1}. We are done.
\end{proof}
Let us now turn to the proof of Lemma \ref{lemma_reduce integration support of central estimate}. Note that by Theorem \ref{theorem_ball-to-band} it can be reduced to the following.
\begin{prop}\label{theorem_proposition final reduction}
	For all $\varepsilon>0,$ there exists $C_\varepsilon>0$ such that
	\begin{equation}\label{eq_remark_reduce integration support even further}
		\left\|\schtf{y}\right\|_{L^2\left(X_{\lambda,\eta}\cap(B^n_{\lambda}(k)\times[0,2\lambda]),dydt\right)}\lessssim \eta^{-\frac12} R^{s_0(n,\alpha)} \|f\|_2,
	\end{equation}
	for all 
	%	$1\leq \lambda\leq R^{\frac{1-2\alpha}{1-\alpha}}$, 
	$1\leq \lambda\leq R$,
	$1\leq\eta\lesssim R^{1-2\alpha}\lambda^{\alpha-1} $
	and $k\in \lambda\Z^n$.
\end{prop}

\subsection{Proof of the last reduction: Proposition \ref{theorem_proposition final reduction}}
In this section, we want to use Theorem \ref{theorem_Cor1.7 of Du Zhang} to prove Proposition \ref{theorem_proposition final reduction}. 
%By Remark \ref{remark_reduce integration support even further} and Remark \ref{remark_reduce integration support} this would prove \eqref{eq_izquierda t>1}.
%and \eqref{eq_remark_reduce integration support even further - right estimate}.
% The following lemma bounds the density factor arising from said theorem.
Before that, we need to prove three facts related to the density factor $\phi$ that arises from Theorem \ref{theorem_Cor1.7 of Du Zhang}. 

The density factor $\phi$ always has the following three properties. First, fix $1\leq \lambda\leq R$ and observe that
%$1\leq \lambda\leq R$
%choice 
for any
$X'\subset B_R^{n+1}(0)$, a union of lattice unit cubes in $B_R^{n+1}(0)$, we have that
\begin{equation}\label{eq_observation on densities involving tau_0 vs R}
	\phi_{X',n,\lambda}\leq \phi_{X',n,R},
\end{equation}
by definition \eqref{equation_definition of beta density of X}. 

Second, we want to remark that translations of the set $X'$ do not change the estimate. That is, for any choice of $X'\subset B_R^{n+1}(0)$ and any $k\in\Z^n\cap B^n_R(0)$, 
\begin{equation}\label{eq_observation on densities involving translation by (k,0)}
	\phi_{X'-(k,0),n,R} \leq \phi_{X',n,2R}.
\end{equation}
This can be shown as follows. Fix $(y,\tau)\in B^{n+1}_R(0)$. Then
\begin{align}
	&\# \left\{\tildeq \subset X' -(k,0) : \tildeq \subset B_r^{n+1}(y,\tau)\right\}
	\\
	=
	&\# \left\{\tildeq \subset X' : \tildeq \subset B_r^{n+1}(y+k,\tau)\right\}.
	\intertext{However, calling $y+k=y'$, we have $(y',\tau)\in B_{2R}^{n+1}(0)$. Thus the above is}
	&\# \left\{\tildeq \subset X' : \tildeq \subset B_r^{n+1}(y',\tau)\right\}.
\end{align}
Taking a maximum as in \eqref{equation_definition of beta density of X} completes the proof of \eqref{eq_observation on densities involving translation by (k,0)}.

The third property is the Lemma below.

\begin{lemma}\label{lemma_key bound of the phi density}
	For all $1\leq \lambda\leq R$, all $\eta\geq 1$
%	all $1\leq\eta\lesssim R^{1-2\alpha}\lambda^{\alpha-1} $, 
	and all $f\in L^2(\R^d)$, 
	we have that, for $X_{\lambda,\eta}$ defined as above,
	\begin{equation}\label{eq_density phi bound for X-t0-eta - lemma copy}
		\phi_{X_{\lambda,\eta},n,R} \lesssim \eta^{-1} \max\{1,R^{1-2\alpha}\lambda^{\alpha-1}\}.
	\end{equation}
\end{lemma}
%We prove this lemma in Section \ref{section_estimate density phi for our choice of theta(t)}.
\begin{proof}

	Fix $\lambda,\eta$ and $f$.
%	 Construct the associated set $X_{\lambda,\eta}$.
%	%We can now calculate the density of cubes that form $X_{\lambda,\eta}$. 
%	We want to estimate $\phi_{X_{\lambda,\eta},n,R}$. 
	We are going to bound the amount of unit cubes within $X_{\lambda,\eta}$ that can fit inside an arbitrary ball of radius $r\leq R$. 
	Notice that, in order to estimate \eqref{eq_density phi bound for X-t0-eta - lemma copy}, it is enough to consider $r\leq \lambda$. Indeed, for $\lambda \leq r,$ covering $B_r^{n+1}(x',y')\cap X_{\lambda,\eta}$ with about $\left(\frac r\lambda\right)^n$ balls of radius $\lambda$, we see that
	\begin{equation}
		\#\left\{ Q_k\subset X_{\lambda,\eta}: Q_k \subset B_r^{n+1}(x',t')\right\} 
		\leq \left(\frac r\lambda \right)^n \sup\limits_{\substack{
%				B_\lambda^{n+1}(y,\tau)
				y\in \R^n
				\\ \lambda\leq \tau\leq 2\lambda}}
			\#\left\{ Q_k\subset X_{\lambda,\eta}\cap B_\lambda^{n+1}(y,\tau)\right\}.
	\end{equation}
	% Ah ya entiendo. El lado derecho es completar la banda con bolas de máxima densidad en la escala \lambda. Con razón eso tiene que tener más cuadrados (o los mismos si la densidad es alta en toda la banda) que tiene el lado izquierdo.

	Take  $(y,\tau)$ where $\lambda\leq\tau\leq 2\lambda$ and $0< r \leq \lambda$ such that $B_r(y,\tau)\subset B_{4\lambda}^{n+1}(0)$.
%	By the definition of $X_{\lambda,\eta}$, we can assume that $\tau \simeq \lambda$. Thus,
	Then
%	of $X_{\lambda,\eta},$ 
	\begin{align}
		\# \left\{\tildeq \subset X_{\lambda,\eta} : \tildeq \subset B_r^{n+1}(y,\tau) \right\} 
		\leq
		&\eta^{-1} \#\left\{ j\in\Z^n: \tildeq_j \subset B_r^{n+1}(y,\tau)\cap\left(\R^n\times[\lambda,2\lambda]\right) \right\}.
		\label{eq_intermediate step number3 in counting}
	\end{align}
	The set above is a subset of
	\begin{equation}
		\left\{j\in\Z^n : (j+\theta(t),t)\in B^{n+1}_{r+1}\left(y,\tau \right) \text{, for some } t\in[0,R] \right\}.
	\end{equation}
	The cardinal of this set is bounded by the $n$-volume of 
	\begin{align}
		\phantom{=} &\, \left\{ x\in \R^n: (x+\theta(t),t) \in B_{r+1}^{n+1}(y,\tau) \text{, for some } t \in[0,R]\right\}
		\\
		= & \, \left\{x\in \R^n : (x+\theta(t),t) \in B_{r+1}^n(y)\times[\tau-r,\tau+r] \text{, for some } t\in[0,R] \right\}
		\\
		\subset&\, B_{r+1}^n(y)-\theta[\tau-r,\tau+r].
%		\verluegoalphados{=:A.}
	\end{align}
	Whenever $\len(\theta[\tau-r,\tau+r])<r$, the above set has a volume comparable to $r^n$. Otherwise, the above set is a tubular $r$-neighborhood of $\theta[\tau-r,\tau+r]$. Therefore, the $n$-volume of the aforementioned set is, up to dimensional constants, lesser than or equal to $ r^{n-1}\len\left(\theta[\tau-r,\tau+r]\right)$. From these facts and \eqref{eq_intermediate step number3 in counting}, we have that
	\begin{align}\label{eq_bound of tubular r-neighbourhood}
		\# \left\{\tildeq \subset X_{\lambda,\eta} : \tildeq \subset B_r^{n+1}(y,\tau) \right\} 
		\lesssim
		\eta^{-1}\max\left\{r^n,r^{n-1}\len\left(\theta[\tau-r,\tau+r]\right)\right\}
	\end{align}
%	\subsubsection{Calculating the length of $\theta$}
	%The calculation for two variables is analogous to the one for $n$ variables. We offer the simpler calculation here. That is, we take $\theta(t) = (R^{1-2\alpha}t^\alpha,R^{1-2\alpha'}t^{\alpha'})$. Without loss of generality, we can assume 
	%$\alpha \leq \alpha'.$
	
	Let us calculate the value of the length appearing above. Remember we are denoting $\alpha =\min_{\substack{1\leq j \leq n}}\{\alpha_j\}.$ By definition \eqref{eq_definition of theta} of our curve, we have that
	\begin{align}
		\len\left(\theta[\tau-r,\tau+r]\right) 
		&= \int_{\tau-r}^{\tau+r}|\theta'(s)|ds
		\label{eq_len of Gamma first step}
		\simeq
		\int_{\tau-r}^{\tau+r} 
		R^{1-2\alpha}s^{\alpha-1}  ds
		\lesssim
		2rR^{1-2\alpha}\tau^{\alpha-1}
	\end{align}
	Recall
%	\eqref{eq_intermediate step number1 in counting}, 
	\eqref{eq_bound of tubular r-neighbourhood} and use the above to deduce that
\begin{align}
	\# \left\{\tildeq \subset X_{\lambda,\eta} : \tildeq \subset B_r^{n+1}(y,\tau) \right\} 
	\lesssim
	\eta^{-1}r^n\max\left\{ 1, R^{1-2\alpha}\tau^{\alpha-1}\right\}.
	\label{eq_bound the A-set within Rn v2}
\end{align}
%Notice that this bound does not depend on $(y,\tau)$. 
Therefore, by definition \eqref{equation_definition of beta density of X}, since $\tau\simeq \lambda$,
%of $\phi_{X_{\lambda,\eta},n,R} $,
\begin{equation}\label{eq_density phi bound for X-t0-eta}
	\phi_{X_{\lambda,\eta},n,R} 
%	\frac{\# \left\{\tildeq \subset X_{\lambda,\eta} : \tildeq \subset B_r^{n+1}(y,\tau) \right\} }{r^n}
	\lesssim 
	\eta^{-1} \max\{1,R^{1-2\alpha}\lambda^{\alpha-1}\},
\end{equation}
and we are done.
\end{proof}

We now turn to the proof of Proposition \ref{theorem_proposition final reduction}
\begin{proof}
	
Fix 
%$ 1\leq \lambda \leq R^{\frac{1-2\alpha}{1-\alpha}}$ 
$ 1\leq \lambda \leq R$ 
and $1\leq \eta \lesssim R^{1-2\alpha} \lambda^{\alpha-1}$. 
Fix $k\in \Z^n\cap B^{n+1}_R(0)$. Remember that $X_{\lambda,\eta}\subset B_R^{n+1}(0)$ by definition.
Apply \Cref{theorem_Cor1.7 of Du Zhang}, at scale $\lambda$ 
instead of $R$, 
to bound 
\begin{align}
	&\phantom{leq}\eta^{\frac12}\left\|e^{it\Delta}f\right\|
	_{L^2
			\left(X_{\lambda,\eta}\cap B^{n+1}_{\lambda}(k,0)
			%		\cap B^{n+1}_R(0)
			\right)}
	\\&= 
	\eta^{\frac12}\left\|e^{it\verluegocinco{-0}\Delta}[f(\cdot-k)]\right\|
	_{L^2
			\left((X_{\lambda,\eta}-(k,0))\cap B^{n+1}_{\lambda}(0)
			%		\cap B^{n+1}_R(0)
			\right)}
	\\
	&\leq
	C_\varepsilon \lambda^\varepsilon \eta^{\frac12} \left(\phi_{\left(X_{\lambda,\eta}-(k,0)\right)\cap B^{n+1}_{\lambda}(0),n,\lambda}\right)^{\frac 1{n+1}} 
	\lambda^{\frac{n}{2(n+1)}} \|f(\cdot-k)\|_2.
	\intertext{By 
			\eqref{eq_observation on densities involving tau_0 vs R} and \eqref{eq_observation on densities involving translation by (k,0)},
			this is } 
	&\leq 
	C_\varepsilon \lambda^\varepsilon \eta^{\frac12} (\phi_{X_{\lambda,\eta},n,2R})^{\frac 1{n+1}}\lambda^{\frac{n}{2(n+1)}}\|f\|_2.
	\intertext{By Lemma \ref{lemma_key bound of the phi density} the above is}
	&\lesssim 
	C_\varepsilon \lambda^\varepsilon \eta^{\frac12}
	%	(\phi_{X_{\lambda,\eta},n,2R})^{\frac 1{n+1}}
	(\eta^{-1} \max\{1,R^{1-2\alpha}\lambda^{\alpha-1}\})^{\frac 1{n+1}}
	\lambda^{\frac{n}{2(n+1)}}\|f\|_2
\label{eq_step before the split of cases of lambda}
	\intertext{
	Observe that, whenever $R^{\frac{1-2\alpha}{1-\alpha}}\leq \lambda \leq R$, we have $\eta\lesssim R^{1-2\alpha}\lambda^{\alpha-1} \leq 1$.
	}
%	\\
	&
	=
	C_\varepsilon \lambda^\varepsilon
	\eta^{\frac12- \frac{1}{n+1}}
    \lambda^{\frac{n}{2(n+1)}}\|f\|_2
	\\
	&
	\lesssim C_\varepsilon R^\varepsilon R^{\frac n{2(n+1)}}\|f\|_2,
\end{align}
	and \eqref{eq_remark_reduce integration support even further} is established in this case.

	Otherwise, if $1\leq\lambda\leq R^{\frac{1-2\alpha}{1-\alpha}}$,
	by \eqref{eq_step before the split of cases of lambda},
\begin{align}
	&\phantom{leq}\eta^{\frac12}\left\|e^{it\Delta}f\right\|
	_{L^2
		\left(X_{\lambda,\eta}\cap B^{n+1}_{\lambda}(k,0)
		%		\cap B^{n+1}_R(0)
		\right)}
	\\
	&\lessssim 
	\eta^{\frac12- \frac{1}{n+1}}\left(R^{1-2\alpha} \lambda^{\alpha-1}\right)^{\frac1{n+1}} \lambda^{\frac{n}{2(n+1)}}\|f\|_2.
	\intertext{Since $\eta\lesssim R^{1-2\alpha}\lambda^{\alpha-1},$ this is}
	&\leq
	C_\varepsilon R^\varepsilon
	\left(R^{1-2\alpha}\lambda^{\alpha-1}\right)^{\frac12-\frac1{n+1}}
	\left(R^{1-2\alpha}\lambda^{\alpha-1}\right)^{\frac1{n+1}} \lambda^{\frac{n}{2(n+1)}}\|f\|_2
	\\
	&= 
	C_\varepsilon R^\varepsilon
	R^{\frac{1-2\alpha}{2}} \lambda^{\frac\alpha2-\frac 1{2(n+1)}}\|f\|_2 .
\end{align}
Denote the constant above by $C_{R,\lambda}=R^{\frac{1-2\alpha}{2}} \lambda^{\frac\alpha2-\frac 1{2(n+1)}}$. We distinguish two cases. 

In the first case, if $\alpha<\frac1{n+1}$, then $\frac\alpha2 - \frac1{2(n+1)} < 0,$ thus $C_{R,\lambda} \leq R^{\frac{1-2\alpha}{2}}\leq R^{s_0(n,\alpha)}$ and we are done.

In the second case, if $\alpha\geq\frac1{n+1}$, since $\lambda\leq R^{\frac{1-2\alpha}{1-\alpha}}$, 
\begin{align}
	C_{R,\lambda} 
	&\leq R^{\frac{1-2\alpha}{2}}R^{\frac{1-2\alpha}{1-\alpha}\left(\frac\alpha2-\frac1{2(n+1)}\right)}
	\\
	&=
	R^{\frac{1-2\alpha}{1-\alpha}\left(\frac{1-\alpha}2 + \frac\alpha2 - \frac1{2(n+1)}\right)}
	\\
	&=
	R^{\frac{1-2\alpha}{1-\alpha}\cdot\frac{n}{2(n+1)}}
	\leq R^{s_0(n,\alpha)}.
\end{align}
%By Remark \ref{remark_reduce integration support even further}, we have proven estimate \eqref{eq_central-constante combinada}.
This gives \eqref{eq_remark_reduce integration support even further} for $ 1\leq \lambda \leq R^{\frac{1-2\alpha}{1-\alpha}}$. 
\end{proof}

%\section{Bounding the Schrödinger evolution over unbounded space.}\label{section_proof of theorem ball to band}
%\input{Cuts/Section 4 - Th2.4 proof}

\pagebreak

	%%%%%%%%%%%%%% BIBLIOGRAPHY %%%%%%%%%%%%%
	\nocite{*}
	\medskip
	
	\printbibliography
	% Comment this after use.!!!
	%\input{Th1.3.M-main.bbl} % Comment this after use.!!!
	% Comment this after use.!!!
	
	\footnotesize
	\sc JAVIER MINGUILLÓN, DEPARTMENT OF MATHEMATICS, UNIVERSIDAD AUTÓNOMA DE MADRID, 28049 MADRID, SPAIN
	
	\normalfont
	\textit{E-mail address:} javier.minguillon@uam.es
	
	\sc FERNANDO SORIA, DEPARTMENT OF MATHEMATICS, UNIVERSIDAD AUTÓNOMA DE MADRID, 28049 MADRID, SPAIN
	
	\normalfont
	\textit{E-mail address:} fernando.soria@uam.es
	
	\sc ANA VARGAS, DEPARTMENT OF MATHEMATICS, UNIVERSIDAD AUTÓNOMA DE MADRID, 28049 MADRID, SPAIN
	
	\normalfont
	
	\textit{E-mail address:} ana.vargas@uam.es
\end{document}